\documentclass{amsart}
\numberwithin{equation}{section}

\newtheorem{thm}{Theorem}[section]
\newtheorem{lem}[thm]{Lemma}

\newtheorem{remark}[thm]{Remark}

\begin{document}

\begin{center}
\textbf{{\large {\ On determining the fractional exponent of the subdiffusion equation }}}\\[0pt]
\medskip \textbf{Shavkat Alimov$^{1}$}\\[0pt]
\textit{sh\_alimov@mail.ru\\[0pt]}
\medskip \textbf{Ravshan Ashurov$^{2}$}\\[0pt]
\textit{ashurovr@gmail.com\\[0pt]}
\medskip \textit{\ $^{1}$ Mirzo Ulugbek National University of Uzbekistan 4, University St., Tashkent 700174, Uzbekistan}

\medskip \textit{\ $^{2}$ Institute of Mathematics, Academy of Science of Uzbekistan}

\end{center}

\textbf{Abstract}: Determining the unknown order of the fractional derivative in differential equations simulating various processes is an important task of modern applied mathematics. In the last decade, this problem has been actively studied by specialists. A number of interesting results with a certain applied significance were obtained. This paper provides a short overview of the most interesting works in this direction. Next, we consider the problem of determining the order of the fractional derivative in the subdiffusion equation, provided that the elliptic operator included in this equation has at least one negative eigenvalue. An asymptotic formula is obtained according to which, knowing the solution at least at one point of the domain under consideration, the required order can be calculated.


\section{Introduction}
In recent decades, specialists have found many new applications of the theory of differential equations with fractional derivatives in many completely different fields of science and technology (see, [1-5]). At the same time, methods are being developed for solving and establishing various properties of solutions to initial boundary value problems for fractional differential equations (see, e.g., [6-11]).

This work is devoted to the study of inverse problems for fractional order differential equations. Inverse problems in the theory of differential equations of integer and fractional orders are usually called problems in which, along with solving the differential equation, it is also necessary to determine the coefficient(s) of the equation and/or the right-hand side (source function). Note that interest in the study of inverse problems is due to the importance of their applications in many areas of modern science, including mechanics, seismology, medical tomography, epidemics, geophysics, etc. (see for classical differential equations the fundamental monograph by S.I. Kabanikhin [12] and for fractional differential equations - handbook  [1]). A number of works by the authors are devoted to inverse problems of determining the right-hand side of various fractional differential equations (see, e.g., [7], [13-18] and references therein).

In this paper, we study a relatively new inverse problem that arises only when modeling a
process using fractional order differential equations, namely, the inverse problem of determining
the order of a fractional derivative. When considering equations with fractional derivatives, the
order of such derivatives is not always known in advance, and what is very important, there are
no instruments for measuring them. Therefore, in the last 10-15 years, specialists have been
actively studying inverse problems of this type. An overview of works published before 2019 can be found in work [19].

In particular in the following two works Li et al. [20] and [21], the authors proved the uniqueness of the order of fractional derivative for the multi-term time-fractional diffusion equation and distributed order fractional diffusion equations  correspondingly.

In the next two papers [22] and [23] the authors studied inverse problems for the simultaneous determination of two parameters $\rho$ and $\sigma$ for an equation $D^\rho_t u -\Delta^\sigma u=0$, $\rho, \sigma\in (0,1)$, in the $n$-dimensional domain $\Omega$, where $D_t^\rho$ is the Caputo derivative and $\Delta$ is the Laplace operator. It is proven that the data $u(x_0, t)$, $0\leq t\leq T$, at a monitoring point $x_0\in \Omega$ guarantee the uniqueness of both parameters.

Let us make one general remark about all the cited works in [19]. Let us consider the initial-boundary value problem for the subdiffusion equation $D^\rho_tu -\Delta u=0$ in the domain $\Omega$. In order to determine the order of the fractional derivative, the authors used information about the solution of the direct problem in a certain time interval and a fixed spatial variable: $u(x_0, t) =\mu(t)$, $0\leq t\leq T$, $x_0\in \Omega$ and $\mu$ is a given function. In all the works listed in the review article [19], only the uniqueness of the solution to the inverse problem was proven. But on the other hand, by specifying function $\mu(t)$ in an arbitrary manner, it is in principle impossible to prove the existence of parameter $\rho$, since $x_0\in \Omega$, and therefore function $\mu(t)$ must satisfy the equation in question and be infinitely differentiable.

It is in connection with this in the survey paper [19] by Z. Li et al. in the section of Open
Problems the authors noted: ”It would be interesting to
investigate inverse problem by the value of the solution at a fixed time as
the observation data”. Exactly this kind of research was carried out in work [24]: by setting the value of the integral $\int_\Omega u(x, t)dx$ at time instant $t=t_0$, the authors managed to prove not only the uniqueness but also the existence of the unknown parameter $\rho$. Then in work [25] a similar result was obtained for the fractional wave equation, in work [26]- for the Rayleigh-Stokes equation, in work [27]- for systems of fractional pseudodifferential equations, and finally in work  [28] - for fractional equations of mixed type.

In work [29], the inverse problem of simultaneously determining the order of the fractional derivative and the right-hand side of the equations is studied. The authors managed to find additional conditions that guarantee the existence and uniqueness of both unknowns.

Note that the method used in work [24] requires that the first eigenvalue of the elliptic part of the equation be equal to zero. In [30] and [31] another method was proposed and this limitation was removed; the existence and uniqueness of a solution to the inverse problem is proven for an arbitrary subdiffusion equation.

Let us pay attention to the fact that in all the listed works where the inverse problem is studied, the spectrum of the elliptic part of the equations is discrete. The question naturally arises: is it possible to solve the inverse problem in the case when the elliptic part has a continuous spectrum? Work [32] gives a positive answer to this question. Moreover, the authors considered a two-parameter inverse problem similar to the one considered in works [22] and [23]. Unlike those works, here the existence and uniqueness of both unknown parameters is proven.

Let us also note the following recent work Gongsheng Li at al. [33],  where for the first time the unique solvability of the inverse problem was proved with the available measurement, given at a single space-time point $(x_o, t_0)$ with sufficiently large $t_0$. Note that the authors considered the equation $D^\rho_tu -d\Delta u=0$, where $d$ is a sufficiently small number, in the domain $\Omega$ with dimension $n\leq 3$ and assumed that for a sufficiently large $J$ all eigenfunctions $v_j(x)$, $j\leq J$, of the Laplace operator with the Dirichlet condition are non-negative at the point $x_0$ where the additional condition is given and all the Fourier coefficients of the initial function are non-negative.

Let us dwell on the following two works [34] and [35] which are closest to our research, where formulas were obtained for calculating the unknown order of derivatives.

In the paper of Hatano et al. [34]  the
equation $D_t^\rho u =\Delta u$ is considered in $\Omega\subset \mathbb{R}^n$ with the Dirichlet boundary
condition and the initial function $\varphi(x)$. They proved
that if $\varphi\in C_0^\infty(\Omega)$,  $
\varphi(x)\geq 0$ or $\leq 0$ and is not identical to zero on $\overline{\Omega}$, then one can determine the order $\rho$ as
$$\rho =\lim\limits_{t\rightarrow
\infty}\big[tu_t(x_0,t)[u(x_0,t)]^{-1}\big],$$
with an arbitrary $x_0\in \Omega$. From this, in particular, it is clear that, the order of the derivative in this formula does not depend on the choice of $x_0$. If $\Delta
\varphi(x_0)\neq 0$ and $\varphi\in C_0^\infty(\Omega)$, then the authors also proved
that
$$\rho =\lim\limits_{t\rightarrow
0}\big[t u_t(x_0,t)[u(x_0,t)-\varphi(x_0)]^{-1}\big].$$

The authors of the work [35] managed to obtain, in the case of various ordinary fractional differential equations, a formula for calculating the unknown order of the derivative in the form of a similar limit, i.e. $\lim\limits_{t\to \infty} \big((tu')/u\big)$. When establishing this formula, the asymptotic behavior of the Mittag-Leffler functions and the following simple observation were taken as a basis: if we take the function $u(t)=t^\beta$, then it is easy to see that the specified limit is equal to $\beta$.

Despite the simplicity of these formulas, there is, in our opinion, one inconvenience when used in real-live problems. Usually we know the solution to the problem under consideration with some accuracy. And when calculating the derivative of an approximately known solution, we admit new inaccuracies. In connection with this observation, the question naturally arises: is it possible to represent the unknown order of the derivative in the form of a formula without involving the derivative of the solution to the direct problem? 

This work is devoted to the study of precisely this issue. We present a formula (Theorem \ref{T1}) for calculating the parameter $\rho$ under less stringent conditions on the initial function than in previous works. Moreover, when calculating an unknown parameter, there is no need to calculate the derivative of the solution to the initial boundary value problem.
\section{Main result}

The basis of many modern mathematical models describing the emergence and spread of the pandemic process is the subdiffusion equation (see, e.g., [36-38]). It is obtained from the diffusion equation by replacing the time derivative with a fractional derivative of order $\rho$, where $0<\rho<1$. Based on data [39] presented by Johns Hopkins University on outbreaks in different countries and  on numerous observations, it was found that this equation more adequately describes the process being studied and makes it possible to predict its development with greater accuracy (see, e.g., [40]). At the same time, there is no unanimity on the question of what the order $\rho$  of the fractional derivative should be. 

This paper proposes a solution to this inverse problem of finding the order $\rho$ in the case when diffusion does not lead to attenuation, but to growth over time as a geometric progression, which is typical for the pandemic process.

In this work, we understand the fractional derivative in the Riemann-Liouville sense, but the results of the work are also valid in the case of derivatives in the Caputo sense. Let us recall the corresponding definitions (see, e.g., [2], p. 92).

The Riemann-Liouville fractional  integral of order $\rho<0$ , of a
function $f$ defined on $[0,
\infty)$
is determined by the formula
\[
J_t^\rho f(t)=\frac{1}{\Gamma
(-\rho)}\int\limits_0^t\frac{f(\xi)}{(t-\xi)^{\rho+1}} d\xi, \quad
t>0,
\]
provided that the right-hand side exists. Here $\Gamma(z)$ is
the gamma function. Using this definition, one can define the Riemann-Liouville fractional derivative of order $\rho\in (0,1)$ as 
$$
\partial_t^\rho f(t)= \frac{d}{dt}J_t^{\rho-1} f(t).
$$
If we swap the integral and derivative in this equality, we obtain the definition of a derivative in the sense of Caputo of order $\rho$:
$$
D_t^\rho f(t)= J_t^{\rho-1} \frac{d}{dt}f(t).
$$
Note that if $\rho=1$, then the both fractional derivatives coincide with
ordinary classical derivative of the first order.

Let $0<\rho<1$. Consider the following Cauchy problem:
\begin{equation}\label{1}
\left\{
\begin{aligned}
&\partial_t^{\rho} u(x,t) + Au(x,t) = 0, \quad t>0,\\
&\underset{t\to 0}{\lim}\ J_t^{\rho-1}u(x,t)  =\ \phi, \\
\end{aligned}
\right.
\end{equation}
where $\phi\in C(\overline{\Omega})$ is the given initial function. Here $\Omega\subset \mathbb{R}^n$ is an arbitrary domain with a smooth boundary, $A:H\to H$ is an elliptic self-adjoint operator in the Hilbert space $H = L_2(\Omega)$ generated by the differential operation
\begin{equation}\label{2}
A(x,D)v(x)\ =\ -\ \sum\limits_{i,j=1}^n \frac{\partial }{\partial x_j}\left(a_{ij}(x) \frac{\partial v}{\partial x_i}\right)\ +\ c(x)v(x), \quad x\in\Omega, 
\end{equation}
and the boundary condition
\begin{equation}\label{3}
\frac{\partial v}{\partial \nu}\ +\ \sigma(x)v(x)\ =\ 0, \quad x\in \partial\Omega. 
\end{equation}
Here $\partial u /\partial \nu$ is the conormal derivative (see, e.g., [41], p. 89, (1.14)). 

We assume that the ellipticity condition is satisfied
\[
\sum\limits_{i,j=1}^n a_{ij}(x) \xi_i\xi_j\ \geq\ a_0 \sum\limits_{j=1}^n |\xi_j|^2, \quad a_0 >0, \quad x\in\overline{\Omega}, \quad \xi\in \mathbb{R}^n,
\]
as well as the usual condition $c(x)\geq0$ for elliptic operators.

From the ellipticity condition it follows that the self-adjoint operator $A$ is semi-bounded and its spectrum consists of eigenvalues
\[
\lambda_1 < \lambda_2 < \lambda_3 < ..., \quad \lambda_k\to+\infty, \quad (k\to+\infty),
\]
each of which has finite multiplicity (see [41], Chapter III, Theorem 3.5).

The spectral expansion of an arbitrary function $f\in L_2(\Omega)$ has the form
\[
f(x)\ =\ \sum\limits_{k=1}^\infty P_k f(x),
\]
where $P_k$ is the projection operator onto the eigensubspace corresponding to the eigenvalue $\lambda_k$.

It is known that for $\alpha>n/4$ the expansion of any function $f\in D(A^\alpha)$ converges absolutely and uniformly in the closed domain $\overline{\Omega}$ (see [42] for the Laplace operator and [43], Chapter 4, 16, for the general operator (\ref{2})). In this case, the inequality
\begin{equation}\label{4}
|f(x)|\ \leq\ \sum\limits_{k=1}^\infty |P_kf(x)|\ \leq\ C\|A^\alpha f\|_{L_2(\Omega)} 
\end{equation}
holds.

Each eigenfunction $v(x)$ corresponding to an eigenvalue $\lambda\in \mathbb{R}$ satisfies the differential equation
\[
A(x,D)v(x)\ =\ \lambda v(x)
\]
and boundary condition (\ref{3}).

For any smooth function $u(x)$ the identity
\[
\int\limits_{\Omega} [A(x,D)u(x)] u(x)\, dx\ =\ \int\limits_\Omega \sum\limits_{i,j=1}^n a_ {ij}(x) \frac{\partial u}{\partial x_i} \frac{\partial u}{\partial x_j}\, dx
\]
\[
+\ \int\limits_\Omega c(x) |u(x)|^2\, dx\ -\ \int\limits_{\partial\Omega} u(x)\frac{\partial u}{\partial \nu}\,ds(x)
\]
 holds, obtained by integration by parts (see, e.g., [41], p. 90). This implies for each eigenfunction $v(x)$ the equality
\begin{equation}\label{5}
\lambda\ =\ \int\limits_\Omega \sum\limits_{i,j=1}^n a_{ij}(x) \frac{\partial v}{\partial x_i} \frac{\partial v} {\partial x_j}\, dx\ +\ \int\limits_\Omega c(x) |v(x)|^2\, dx\ +\ \int\limits_{\partial\Omega} |v(x) |^2 \sigma(x)\, ds(x).
\end{equation}

Let us divide the boundary $\partial\Omega$ of the domain $\Omega$ into the following three parts:
\[
E^+(\sigma)\ =\ \{x\in\partial\Omega\ |\ \sigma(x)>0\},
\]
\[
E^-(\sigma)\ =\ \{x\in\partial\Omega\ |\ \sigma(x)<0\},
\]
\[
E^0(\sigma)\ =\ \{x\in\partial\Omega\ |\ \sigma(x)=0\}.
\]
The flow exits through the surface $E^+(\sigma)$, the flow enters through $E^-(\sigma)$, and at the points $E^0(\sigma)$ the flow is zero.

Note that in the case when $\text{mes}_{n-1}\,E^-(\sigma)>0$, i.e., the surface measure of the set $E^-(\sigma)$ is positive, it follows from equality (\ref{5}) that the eigenvalue may turn out to be negative. This is determined by the values of the coefficient $\sigma(x)$ included in the boundary condition (\ref{3}). At the same time, exponentially growing solutions appear, which can be interpreted as the emergence of a pandemic. This paper examines exactly this case.

Further, as shown below, the solution growth index significantly depends on $|\lambda_1|$, namely, for $|\lambda_1|>1$ it decreases with increasing $\rho$, for $|\lambda_1|<1$ it increases with $\rho$, and for $|\lambda_1|=1$ does not depend on $\rho$ at all. In other words, for $|\lambda_1| = 1$ solutions of the Cauchy problem (\ref{1}) for different $\rho$ behave essentially the same. Thus, to uniquely restore the fractional order $\rho$ from the leading term of the asymptotics, it is necessary to require that the condition $|\lambda_1|\neq 1$ be satisfied.

So our basic assumption is this:
\begin{equation}\label{6}
\lambda_1 < 0, \quad \lambda_1 \neq - 1. 
\end{equation}

Let us denote by the symbol $N_k(\phi)$ the zeros of the function $P_k\phi$:
\begin{equation}\label{7}
N_k(\phi)\ =\ \{x\in\Omega\ |\ P_k \phi(x) = 0\},
\end{equation}
and let $I$ be the identity operator. Fix an arbitrary number $\lambda> |\lambda_1|$.

The following statement is true.

\begin{thm}\label{T1}
    Let condition (\ref{6}) be satisfied and let $0<\rho<1$. Then for any function $\phi\in D((A+\lambda I)^\alpha)$, where $\alpha>n/4$, at each point $x\in\Omega\setminus N_1(\phi)$ the following equality 
\begin{equation}\label{8}
\frac{1}{\ln|\lambda_1|}\ \underset{t\to+\infty}{\lim}\ \left(\ln\ln |u(x,t)| - \ln t\right )\ =\ \frac{1}{\rho} 
\end{equation}
holds.
\end{thm}

\begin{remark}\label{R1} Below, during the proof of Theorem \ref{T1}, we will show that for sufficiently large values of $t$ the inequality $|u(x,t)|>1$ holds.
\end{remark}

\section{Proof of Theorem \ref{T1}}

The proof of Theorem \ref{T1} is based on the spectral expansion of the solution to problem (\ref{1}) in terms of the eigenfunctions of the operator $A$. The Mittag-Leffler function, defined as (see, e.g., [2], p. 42)
\[
E_{\rho,\mu}(t)\ =\ \sum\limits_{m=0}^\infty \frac{t^m}{\Gamma(\rho m+\mu)}
\]
plays an important role in this expansion.

Using this function, the solution to problem (\ref{1}) can be written in the form (see, e.g., [18])
\begin{equation}\label{9}
u(x,t)\ =\ \sum\limits_{k=1}^\infty t^{\rho-1} E_{\rho,\rho}(-\lambda_k t^\rho)P_k \phi( x).
\end{equation}

Let us consider in more detail the first term of expansion (\ref{9}), taking into account that $\lambda_1<0$:
\[
u_1(x,t)\ =\ t^{\rho-1} E_{\rho,\rho}(|\lambda_1| t^\rho) P_1 \phi(x).
\]

\begin{lem}\label{L1} Under the conditions of Theorem \ref{T1}, the equality
\begin{equation}\label{10}
\ln |u_1(x,t)|\ =\ t|\lambda_1|^{1/\rho} \left[1 + o(1)\right], \quad t\to+\infty, \end{equation}
is satisfied.
\end{lem}

\begin{proof}
    Let us denote $z = |\lambda_1| t^\rho$. Then, taking into account the formula
\begin{equation}\label{11}
E_{\rho,\rho}(z)\ =\ \frac1\rho z^{1/\rho - 1} e^{z^{1/\rho}}\ +\ O(1/z^2 ), \quad z>1, 
\end{equation}
(see [11], (1.1.13)), we get
\[
u_1(x,t)\ =\ t^{\rho-1} E_{\rho,\rho}(z) P_1 \phi(x)\ =\ t^{\rho-1} \frac1\rho z ^{1/\rho - 1} e^{z^{1/\rho}} P_1 \phi(x)\ +\ O(1/z^2), \quad z>1.
\]
Hence,
\[
u_1(x,t)\ =\ \frac1\rho\, |\lambda_1|^{1/\rho - 1}\ e^{\,t|\lambda_1|^{1/\rho}} P_1 \phi (x)\ +\ O(1/t^2), \quad t\to+\infty.
\]

Let us put
\[
\beta(x)\ =\ \frac1\rho\, |\lambda_1|^{1/\rho - 1} |P_1 \phi(x)|.
\]
From definition (\ref{7}) and from the conditions of Theorem \ref{T1} it follows that $P_1\phi(x) \neq 0$. Therefore, $\beta(x) > 0$, and hence
\begin{equation}\label{12}
|u_1(x,t)|\ =\ \beta(x)\ e^{\,t|\lambda_1|^{1/\rho}}\left[1 + O(t^{-2}) e ^{\,-t|\lambda_1|^{1/\rho}}\right], \quad t\to+\infty, 
\end{equation}
or
\[
\ln |u_1(x,t)|\ =\ t|\lambda_1|^{1/\rho}\ +\ \ln\beta(x)\ +\ \ln\left[1 + o(1)\right]\ =\ t|\lambda_1|^{1/\rho}\ +\ O(1).
\]
The last equality can be written as follows:
\[
\ln |u_1(x,t)|\ =\ t|\lambda_1|^{1/\rho}\left[1\ +\ \frac{O(1)}{t|\lambda_1|^{1/ \rho}}\right].
\]
This implies the required estimate (\ref{10}).
\end{proof}
\begin{lem}\label{L2}Under the conditions of Theorem \ref{T1}, the estimate
\begin{equation}\label{13}
u(x,t)\ =\ u_1(x,t) \left[1 +\ o(1)\right],\quad x\in \Omega\setminus N_1(\phi), \quad t\to+\infty, 
\end{equation}
is satisfied.   
\end{lem}
\begin{proof} Consider the representation of the solution to problem (\ref{1}) in the form of expansion (\ref{9}). Let us assume that the eigenvalues (ignoring multiplicity) of $\lambda_k$ with numbers $k\leq m$ are negative. In accordance with this, the value of the expansion (\ref{9}) is the sum of three terms:
\begin{equation}\label{14}
u(x,t)\ =\ u_1(x,t)\ +\ \Sigma'\ +\ \Sigma'', 
\end{equation}
where
\[
\Sigma'\ =\ \sum\limits_{k=2}^m t^{\rho-1} E_{\rho,\rho}(-\lambda_k t^\rho)P_k \phi(x),
\]
and
\[
\Sigma''\ =\ \sum\limits_{k=m+1}^\infty t^{\rho-1} E_{\rho,\rho}(-\lambda_k t^\rho)P_k \phi( x).
\]

To estimate the sum $\Sigma'$, consider two cases.

1) Let us first assume that $\lambda_2 < 0$. Then $m\geq2$. Let us show that there is $T>0$ such that the estimate
\begin{equation}\label{15}
|\Sigma'|\ \leq\ Ce^{t|\lambda_2|^{1/\rho}} \|A^\alpha\phi\|, \quad t \geq T, 
\end{equation}
holds. Taking into account that $\lambda_k<0$ for $k\leq m$, from (\ref{11}) we get:
\[
|\Sigma'|\ \leq\ t^{\rho-1} \sum\limits_{k=2}^m E_{\rho,\rho}(-\lambda_k t^\rho)|P_k \phi( x)|
\]
\[
\leq\ \frac{C}{\rho}\, \sum\limits_{k=2}^m |\lambda_k|^{1/\rho - 1} e^{t|\lambda_k|^{1/ \rho}}|P_k \phi(x)|\ \leq\ \frac{C}{\rho}\,|\lambda_2|^{1/\rho - 1} e^{t|\lambda_2|^{ 1/\rho}} \sum\limits_{k=2}^m |P_k \phi(x)|.
\]
Now apply (\ref{4}) to obtain estimate (\ref{15}).

2) Let now $\lambda_2\geq0$. In this case $\Sigma' = 0$ and we can formally write
\begin{equation}\label{16}
|\Sigma'|\ \leq\ C, \quad t\geq 0. 
\end{equation}
Let us introduce the notation
\[
\epsilon\ =\ \begin{cases} |\lambda_1|^{1/\rho} - |\lambda_2|^{1/\rho}, & \text{if} \quad \lambda_2 < 0,\\
|\lambda_1|^{1/\rho}, & \text{if} \quad \lambda_2 \geq 0.\\
\end{cases}
\]
Then, combining estimates (\ref{15}) and (\ref{16}), we obtain the inequality
\begin{equation}\label{17}
|\Sigma'|\ \leq\ Ce^{t|\lambda_1|^{1/\rho}} e^{-\epsilon t}, \quad t \geq T, 
\end{equation}
valid for any value of $\lambda_2$.
To estimate $\Sigma''$ we use the estimate (see [11], (1.1.14))
\[
E_{\rho,\rho}(-z)\ =\ \frac{z^{-2}}{|\Gamma(-\rho)|}\ +\ \frac{O(1)}{z^ 3}, \quad z\to+\infty .
\]
Hence,
\[
|E_{\rho,\rho}(-z)|\ \leq\ C, \quad z\geq 0.
\]
Then from (\ref{4}) we obtain
\[\
|\Sigma''|\ \leq\ t^{\rho-1} \sum\limits_{k=m+1}^\infty \big|E_{\rho,\rho}(-\lambda_k t^\rho)\big| |P_k \phi(x)|
\]
\begin{equation}\label{18}
\leq\ C t^{\rho-1} \sum\limits_{k=m+1}^\infty |P_k \phi(x)|\ \leq\ C t^{\rho-1} \|A ^{\alpha}\phi\|. 
\end{equation}
Thus, by (\ref{17}) and (\ref{18}) one has
\begin{equation}\label{19}
|\Sigma' + \Sigma''|\ \leq\ C e^{t|\lambda_1|^{1/\rho}} e^{-\epsilon t}, \quad t\geq T. 
\end{equation}

Next, we rewrite expansion (\ref{14}) as follows:
\begin{equation}\label{20}
u(x,t)\ =\ u_1(x,t)\left[1\ +\ \frac{\Sigma' + \Sigma''}{u_1(x,t)}\right]. 
\end{equation}
Note that from (\ref{12}) and from the condition $x\in \Omega\setminus N_1(\phi)$ the inequality 
\[
|u_1(x,t)|\ \geq\ C(x) e^{\,t|\lambda_1|^{1/\rho}}, \quad C(x) >0, \quad t\geq 0 ,
\]
follows. Then, from (\ref{19}) and (\ref{20}) we obtain
\[
	u(x,t)\ =\ u_1(x,t)\left[1\ +\ O(e^{-\epsilon t})\right] .
	\]
This implies the required estimate (\ref{13}).
\end{proof}

Now we are ready to prove Theorem \ref{T1}. First of all, note that Lemma \ref{L2} implies the estimate
\[
\ln |u(x,t)|\ =\ \ln |u_1(x,t)|\ +\ o(1),\quad x\in \Omega\setminus N_1(\phi), \quad t\to+\infty.
\]
Further, from Lemma \ref{L1} it follows that there is $T>0$ such that
\[
\ln |u_1(x,t)|\ >\ 1, \quad t\geq T.
\]
Hence,
\[
\ln |u(x,t)|\ =\ \ln |u_1(x,t)|\left[1 +\ \frac{o(1)}{\ln |u_1(x,t)|}\right]\ =\ \ln |u_1(x,t)|\left[1 +\ o(1)\right].
\]
Applying Lemma \ref{L1} again, we get
\[
\ln |u(x,t)|\ =\ t|\lambda_1|^{1/\rho} \left[1 + o(1)\right]\left[1 +\ o(1)\right] \ =\ t|\lambda_1|^{1/\rho} \left[1 + o(1)\right].
\]
Then, taking the logarithm of both sides of this equality, we obtain:
\begin{equation}\label{21}
\ln\ln |u(x,t)|\ =\ \ln t\ +\ \frac1\rho\, \ln|\lambda_1|\ +\ o(1),\quad x\in \Omega\setminus N_1(\phi), \quad t\to+\infty. 
\end{equation}
This implies the required relation (\ref{8}). Theorem \ref{T1} is proved.

\section{Example} Consider a three-dimensional cylindrical domain $\Omega$ of the following form:
\[
\Omega\ =\ D\times (0,H),
\]
where $D$ is a flat domain with piecewise smooth boundary $\partial D$. Let us consider the following problem in this domain
\begin{equation}\label{22}
\partial_t^\rho u(x,y,z,t)\ =\ \Delta u(x,y,z,t), \quad (x,y)\in D, \quad 0<z<h, \quad t >0, 
\end{equation}
where $\rho\in (0,1)$.
The function $u(x,y,z,t)$ is assumed to satisfy the initial condition
\begin{equation}\label{23}
\underset{t\to 0}{\lim}\ J_t^{\rho-1}u(x,y,z,t)  =\  \phi(x,y,z),
\end{equation}
and boundary conditions
\begin{equation}\label{24}
\frac{\partial u}{\partial n}(x,y,z,t)\ +\ \sigma u(x,y,z,t) = 0, \quad (x,y)\in \partial D,\ 0<z<H, \ t>0,
\end{equation}
and
\begin{equation}\label{25}
\frac{\partial u}{\partial z}(x,y,0,t) = 0, \quad \frac{\partial u}{\partial z}(x,y,H,t) - hv( x,y,H,t) = 0, \quad (x,y)\in D, \quad t>0. 
\end{equation}
In conditions (\ref{25}) it is assumed that $h>0$.

Condition (\ref{24}) means that through the cylindrical surface $\partial D\times[0,H]$ located above the boundary of the domain $D$, leakage occurs with a flow proportional to the concentration at the boundary, while, according to condition (\ref{25}), through the surface $S = \{(x,y)\in D, z=H\}$ a similar diffusion occurs, directed inside the domain $\Omega$ (see [44], Chapter III,  1). Exponential growth leading to a pandemic occurs when the relationships between $\sigma$ and $h$ cause a negative eigenvalue to appear.

Let us assume that $\sigma = 0$. This means that the boundary $\partial D\times [0,H]$ is closed and the flow through it is zero. Condition (\ref{24}) then takes the form
\begin{equation}\label{26}
\frac{\partial u}{\partial n}(x,y,z,t)\ =\ 0, \quad (x,y,z)\in \partial D\times (0,H), \ t>0. 
\end{equation}

Let us consider the spectral problem in the two-dimensional domain $D$:
\[
\Delta v(x,y) + \mu v(x,y)\ =\ 0, \quad (x,y)\in D, \quad \frac{\partial v}{\partial n}(x,y) = 0, \ (x,y)\in \partial D.
\]

Let $\mu_1, \mu_2, ...\ -$ be a sequence of eigenvalues, and $v_1(x,y)$, $v_2(x,y), ...\ -$ be the corresponding sequence of eigenfunctions.

Obviously, the first eigenvalue $\mu_1=0$ is prime (i.e., its multiplicity is 1), and the first eigenfunction is equal to $v_1(x,y) = 1/\sqrt{|D|}$.

Next, consider the eigenvalue problem on the interval $[0,H]$ with the spectral parameter $\nu$:
\begin{equation}\label{27}
\left\{
\begin{aligned}
&w''(z) + \nu w(z)  = 0, \quad 0<z<H,\\
&w'(0)  = 0,\\
&w'(H) - h w(H)  = 0.\\
\end{aligned}
\right.
\end{equation}

Positive solutions $\nu_1, \nu_2, ...$  of equations
\[
\tan(H\sqrt{\nu})\ = - \ \frac{h}{\sqrt{\nu}}
\]
are eigenvalues, and the corresponding eigenfunctions have the form
\[
w_j(z) = M_j\cos z\sqrt{\nu_j},
\]
where $M_j$ are positive normalizing factors.

Spectral problem (\ref{27}) also has a unique negative eigenvalue $\nu_0<0$, which is a solution to the equation
\[
\tanh (H\sqrt{-\nu_0}) = \frac{h}{\sqrt{-\nu_0}}.
\]
The function $w_0(z) = M_0 \cosh z \sqrt{-\nu_0}>0$ is an eigenfunction corresponding to the negative eigenvalue $\nu_0$, where $M_0$ is a positive normalizing factor.

The orthonormal system $\{w_j(z)\}_{j=0}^\infty$ is complete in $L_2(0,H)$ (see [45]).

In this case, the self-adjoint operator $A = - \Delta$, generated by the boundary conditions (\ref{25}) and (\ref{26}), has the following complete orthonormal system of eigenfunctions and sequence of eigenvalues:
\[
u_{ij}(x,y,z) = v_i(x,y)\cdot w_j(z), \quad \lambda_{ij} = \mu_i + \nu_j \quad i=1,2,\cdots, \quad j=0,1,2,\cdots.
\]

Note that the first eigenvalue of the operator $A$ is negative and equal to
\[
\lambda_{10} = \mu_1 + \nu_0= \nu_0<0,
\]
and the corresponding eigenfunction is positive and equal to
\[
u_{10}(x,y,z) = v_1(x,y)\cdot w_0(z) = \frac{M_0}{\sqrt{|D|}} \cosh z \sqrt{-\nu_0} > 0.
\]

The symbol $u(x,y,z,t)$ denotes the solution to the Cauchy problem (\ref{22})-(\ref{23}) satisfying the boundary conditions (\ref{25}) and (\ref{26}).

The following statement is true.

\begin{thm}\label{T2} Let $0<\rho<1$ and let $\phi\in W_2^2(\Omega)$ satisfy conditions (\ref{25}) and (\ref{26}) and be positive in the domain $\Omega $. Then for any point $(x,y,z)\in\Omega$ the equality 
\begin{equation}\label{28}
\underset{t\to+\infty}{\lim}\ \left(\ln\ln |u(x,y,z,t)| - \ln t\right)\ =\ \frac{\ln|\lambda_ {10}|}{\rho}\  
\end{equation}
holds.
\end{thm}

\begin{proof} Let us make sure that all the conditions of Lemma \ref{L2} are satisfied. First of all, we note that the function $\phi$ satisfying the conditions of Theorem \ref{T2} belongs to the domain of definition of the operator $A^\alpha$, where $\alpha = 1 $. Since the domain $\Omega$ is three-dimensional, the condition $\alpha > n/4$ is satisfied.

Further, the first eigenfunction $v_{1}(x,y,z)$ is positive, so the multiplicity of the first eigenvalue is 1. Then for any function $\phi\in L_2(\Omega)$ the equality
\[
P_1\phi(x,y,z)\ =\ (\phi,v_1)v_1(x,y,z)
\]
holds.

If $\phi(x,y,z)>0$, then $(\phi,v_1)>0$, and therefore
\[
P_1\phi(x,y,z)\ >\ 0, \quad (x,y,z)\in\Omega.
\]
Consequently, $N_1(\phi) = \oslash$, i.e. $\Omega\setminus N_1(\phi) = \Omega$ (see definition (\ref{7})).

Since all the conditions of Lemma \ref{2} are satisfied, we can use the asymptotic equality (\ref{21}). From this equality the required relation (\ref{28}) follows.
\end{proof}
\begin{remark} Obviously, instead of the condition $\phi(x,y,z)>0$, it is sufficient to require the inequality $(\phi, v_1)\neq 0$.
\end{remark}
\begin{remark}  Equality (\ref{28}), as well as equality (\ref{8}), is also satisfied for $\rho=1$ (see [44], Chapter VI, 2, p. 516), and even when $\lambda_1$ is an arbitrary real number. However, in the case $0<\rho<1$ considered by us, the requirement $\lambda_1<0$ is essential for equality (\ref{28}) to be fulfilled.
\end{remark}
\

\section{Conclusions}

In this paper, we study the inverse problem of determining
the order of a fractional derivative in subdiffusion equations. The subdiffusion equation has proven to be extremely important for more accurately modeling many physical phenomena and processes (for example, the spread of a pandemic process , see, e.g., [36]-[38], [40]). Determining the order $\rho$ of the fractional differential operator in this equation is definitely critical for proper modeling of anomalous diffusion. Since there is no tool to measure this parameter, the inverse problem of finding the unknown order should be considered. In work [34] to calculate $\rho$, the formula $\rho= \lim\limits_{t\to \infty} \big((tu_t)/u\big)$ is proposed, in which the derivative of the solution to the direct problem is involved.

In this work, when condition (\ref{6}) is satisfied, another formula for $\rho$ is obtained, without the participation of the derivative $u_t$ (Theorem \ref{T1}). A fairly natural example of an initial boundary value problem is given, where condition (\ref{6}) holds. It should also be noted that the formula we have proved holds under less restrictive conditions on the initial function than in work [34].


\end{document}